%% file: arxiv_main.tex
\documentclass{article}
\usepackage{afterpage}
\usepackage{algorithm}
\usepackage{algorithmic}
\usepackage{amsmath}
\usepackage{graphicx} 
\usepackage{subfigure}
\usepackage{natbib,url}
\usepackage{float}
\floatstyle{plain}
\restylefloat{figure}

\input defs.tex

\def\Epsilon{\mathcal{E}}

\begin{document}

\title{Robust sketching for multiple square-root LASSO problems}
\author{Vu Pham, Laurent El Ghaoui, Arturo Fernandez\\
University of California, Berkeley\\
\texttt{\{vu,elghaoui\}@eecs.berkeley.edu, arturof@berkeley.edu}}
\date{}
\maketitle

\begin{abstract}
Many learning tasks, such as cross-validation, parameter search, or leave-one-out analysis, involve multiple instances of similar problems, each instance sharing a large part of learning data with the others. We introduce a robust framework for solving multiple square-root LASSO problems, based on a sketch of the learning data that uses low-rank approximations. Our approach allows a dramatic reduction in computational effort, in effect reducing the number of observations from $m$ (the number of observations to start with) to $k$ (the number of singular values retained in the low-rank model), while not sacrificing---sometimes even improving---the statistical performance. Theoretical analysis, as well as  numerical experiments on both synthetic and real data, illustrate the efficiency of the method in large scale applications.
\end{abstract}

\section{Introduction}
In many practical applications, learning tasks arise not in isolation, but as multiple instances of similar problems. A typical instance is when the same problem has to be solved, but with many different values of a regularization parameter. Cross-validation also involves a set of learning problems where the different ``design matrices'' are very close to each other, all being a low-rank perturbation of the same data matrix. Other examples of such multiple instances arise in sparse inverse covariance estimation with the LASSO~(\cite{friedman2008sparse}), or in robust subspace clustering~(\cite{soltanolkotabi2014robust}). In such applications, it makes sense to spend processing time on the common part of the problems, in order to compress it in certain way, and speed up the overall computation.

In this paper we propose an approach to multiple-instance square root LASSO based on ``robust sketching'', where the data matrix of an optimization problem is approximated by a sketch, that is, a simpler matrix that preserves some property of interest, and on which computations can be performed much faster than with the original.
Our focus is a square-root LASSO problem:
\begin{equation}\label{eq:original-lasso}
\begin{array}{ll}
\dsp \min_{w \in \reals^n} \left \| X^T w - y \right \|_2 + \lambda \| w \|_1
\end{array}
\end{equation}
where $X \in \reals^{n \times m}$ and $y \in \reals^{m}$. Square-root LASSO has pivotal recovery properties; also, solving a square-root LASSO problem is as fast as solving an equivalent LASSO problem with both first-order and second order methods~(\cite{belloni2011square}). We chose the square-root version of the LASSO to make the derivations simpler; these derivations can also be adapted to the original LASSO problem, in which the loss function is squared.

In real-life data sets, the number of features $n$ and the number of observations $m$ can be both very large. A key observation is that real data often has structure that we can exploit. Figure~\ref{fig:singular_values_from_real_data} shows that real-life text data sets are often low-rank, or can be well-approximated by low-rank structures.

\afterpage{
\begin{figure}[h]
\begin{center}
\includegraphics[width=0.36\textwidth]{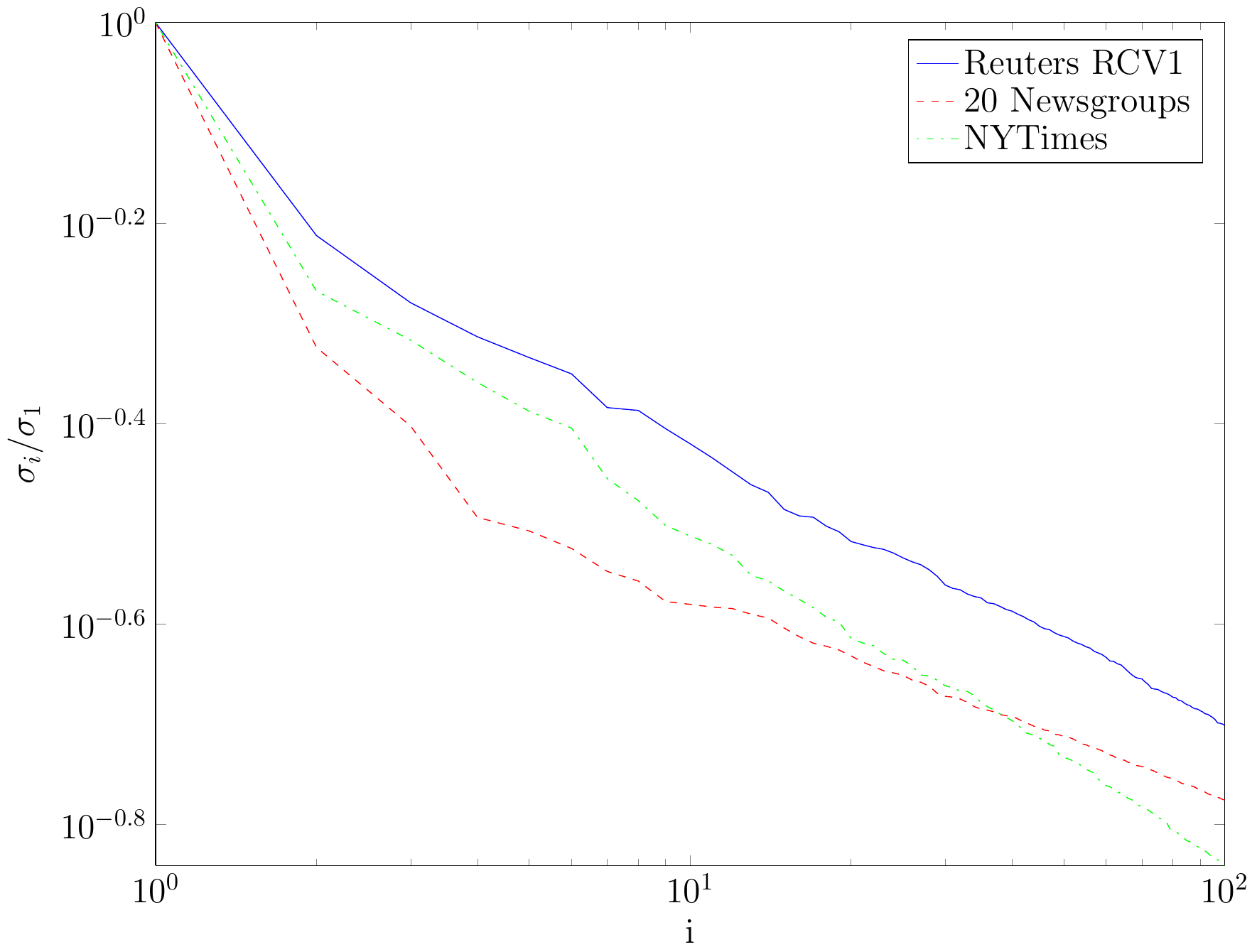}
\end{center}
\caption{Graphs of the top 100 singular values from real-life text data sets.}
\label{fig:singular_values_from_real_data}
\end{figure}
}


\textbf{Contribution.} Our objective is to solve multiple instances of square-root LASSO fast, each instance being a small modification to the same design matrix. Our approach is to first spend computational efforts in finding a low-rank sketch of the full data. With this sketch, we propose a robust model that takes into account the approximation error, and explain how to solve that approximate problem one order of magnitude faster, in effect reducing the number of observations from $m$ (the number of observations to start with) to $k$ (the number of singular values retained in the low-rank model). Together with our proposed model, we can perform cross validation, for example, an order of magnitude faster than the traditional method, with the sketching computation included in our approach. 

This paper employs low-rank sketching for data approximation phase, for which an extensive body of algorithmic knowledge exists, including power method, random projection and random sampling, or Nystr{\"o}m methods (\cite{Miranian20031, drineas2005nystrom, drineas2006fast, Halko:2011:FSR:2078879.2078881, mahoney2011randomized, liberty2013simple}). Our framework works with any approximation algorithms, thus provides flexibility when working with different types of sketching methods, and remains highly scalable in learning tasks.

\textbf{Related work.} Solving multiple learning problems has been widely studied in the literature, mostly in the problem of computing the regularization path (\cite{park2007l1}). The main task in this problem is to compute the full solutions under different regularization parameters. The most popular approach includes the warm-start technique, which was first proposed in specific optimization algorithms (e.g. \cite{yildirim2002warm}), then applied in various statistical learning models, for example in (\cite{kim2007interior, koh2007interior, garrigues2009homotopy}). Recent works (\cite{tsai2014incremental}) show strong interest in incremental and decremental training, and employ the same warm-start technique. These techniques are all very specific to the multiple learning task at hand, and require developing a specific algorithm for each case.

In our approach, we propose a generic, robust, and algorithm-independent model for solving multiple LASSO problems fast. Our model can therefore be implemented with any generic convex solver, providing theoretical guarantees in computational savings while not sacrificing statistical performance.

\textbf{Organization.} The structure of our paper is as follows. In Section~\ref{sec:robust-square-root-LASSO} we propose the robust square-root LASSO with sketched data. In Section~\ref{sec:safe-feature-elimination} we present a simple method to reduce the dimension of the problem. Section~\ref{sec:non-robust-square-root-lasso} studies the effects of a non-robust framework. Section~\ref{sec:complexity-analysis} provides a theoretical complexity analysis and we conclude our paper with numerical experiments in Section~\ref{sec:numerical-results}.

\section{Robust Square-root LASSO}
\label{sec:robust-square-root-LASSO}

\paragraph{Low-rank elastic net.} 
\label{par:low_rank_elastic_net}
Assume we are given $\hat{X}$ as a sketch to the data matrix $X$, the robust square-root LASSO is defined as follows:
\begin{equation}\label{eq:lasso-in-robust-form}
\begin{array}{ll}
\phi_{\epsilon, \lambda} (\hat{X}) & := \dsp \min_{w \in \reals^n} \max_{X: \, \| X - \hat{X} \| \leq \epsilon} \: \| X^T w - y \|_2 + \lambda \| w \|_1 \\
& = \dsp \min_{w \in \reals^n} \max_{\| \Delta \| \leq \epsilon} \: \| (\hat{X} + \Delta)^T w - y \|_2 + \lambda \| w \|_1
\end{array}
\end{equation}
where $X, \hat{X} \in \reals^{n \times m}$, $y \in \reals^m, y \neq 0$ and both $\lambda \geq 0, \epsilon \geq 0$ are given as our parameters. \cite{el1997robust} has shown that
\begin{equation*}
\begin{array}{ll}
\dsp \max_{\| \Delta \| \leq \epsilon} \: \| (\hat{X} + \Delta)^T w - y \|_2 \leq \| \hat{X}^T w - y \|_2 + \epsilon \| w \|_2
\end{array}
\end{equation*}
and the equality holds with the choice of $\Delta$ as
\begin{equation*}
\begin{array}{lll}
\Delta & := & \epsilon u v^T \\
u & := & \left\{\begin{matrix}
\frac{\hat{X}^Tw - y}{ \| \hat{X}^Tw - y \|} & \mbox{if } \hat{X}^Tw \neq y \\
\mbox{any unit-norm vector} & \mbox{otherwise}
\end{matrix}\right. \\
v & := & \left\{\begin{matrix}
\frac{w}{ \| w \|} & \mbox{if } w \neq 0 \\
\mbox{any unit-norm vector} & \mbox{otherwise}
\end{matrix}\right.
\end{array}
\end{equation*}
We also have $\mathbf{rank}(\Delta) = 1$ and $\| \Delta \|_F = \| \Delta \| = 1$, which implies $\Delta$ is the worst-case perturbation for both the Frobenius and maximum singular value norm. Problem ($\ref{eq:lasso-in-robust-form}$) can therefore be rewritten as:
\begin{equation}\label{eq:robust-lasso}
\begin{array}{ll}
\dsp \phi_{\epsilon,\lambda}(\hat{X}) = \dsp\min_{w \in \reals^n} \left \| \hat{X}^T w - y \right \|_2 + \epsilon \| w \|_2 + \lambda \| w \|_1
\end{array}
\end{equation}
Note the presence of an ``elastic net'' term, directly imputable to the error on the design matrix.

In our model, we employ the low-rank approximation of the original data matrix from any sketching algorithm: $\hat{X} = P Q^T$, where $P \in \reals^{n \times k}$, $Q \in \reals^{m \times k}$, $P$ and $Q$ have full rank $k$ with $k \ll \min \{m, n \}$. When the full data matrix $X \in \reals^{n \times m}$ is approximated by $X \simeq P Q^T$, for leave-one-out analysis, the low rank approximation of the ``design matrix'' can be quickly computed by: $X_{\setminus i} \simeq P Q_{\setminus i}^T$ where $\setminus i$ means leaving out the $i$-th observation.

\paragraph{Solving the problem fast.} 
\label{par:solving_the_problem_fast}
We now turn to a fast solution to the low-rank elastic net problem~(\ref{eq:robust-lasso}), with $\hat{X} = PQ^T$. This ``primal'' problem is convex, and its dual is:
\begin{eqnarray*}
\lefteqn{\dsp \phi_{\lambda, \epsilon} =}\\
&& \dsp\min_{w, z \in \reals^n} \left \| Q z - y \right \|_2 + \epsilon \| w \|_2 + \lambda \| w \|_1 \; : \; z = P^T w \\
&= & \dsp\min_{w, z \in \reals^n} \max_{u \in \reals^k} \left \| Q z - y \right \|_2 + \epsilon \| w \|_2 + \lambda \| w \|_1 \\
&& + u^T (z - P^T w) \\
&= & \dsp \max_{u \in \reals^k} \min_{w, z \in \reals^n} \left \| Q z - y \right \|_2 + u^T z + \\
&& \epsilon \| w \|_2 + \lambda \| w \|_1 - (P u)^T w \\
&= & \dsp \max_{u \in \reals^k} f_1(u) + f_2(u) ,
\end{eqnarray*}
where $f_1(u) := \dsp \min_{z \in \reals^n} \left \| Q z - y \right \|_2 + u^T z$ and \\
$f_2(u) := \dsp \min_{w \in \reals^n} \epsilon \| w \|_2 + \lambda \| w \|_1 - (P u)^T w$.

\emph{First subproblem.} Consider the first term in $f_1(u)$:
\begin{equation*}
\begin{array}{lll}
\dsp \| Qz - y \|^2_2 & = & z^T Q^T Q z - 2 y^T Qz + y^T y \\
& = & \bar{z}^T \bar{z} + 2 c^T \bar{z} + y^T y \\
& & \mbox{where} \; \bar{z} := (Q^T Q)^{1/2} z \in \reals^n \\
& & \mbox{and} \; c := (Q^T Q)^{-1/2} Q^T y \in \reals^k \\
& = & \| \bar{z} - c \|^2_2 + y^T y - c^T c.
\end{array}
\end{equation*}
Note that $c^T c = y^T Q (Q^T Q)^{-1} Q^T y \leq y^T y$ since $Q (Q^T Q)^{-1} Q^T \preceq I$ is the projection matrix onto $\range(Q)$. Letting $s := \sqrt{y^T y - c^T c} \geq 0$ gives
\begin{equation*}
\begin{array}{lll}
f_1(u) & := & \dsp \min_{z} \| Qz - y \|_2 + u^T z \\
& = & \dsp \min_{z} \sqrt{\| \bar{z} - c \|_2^2 + s^2} + \bar{u}^T \bar{z} \\
& & \mbox{by letting} \; \bar{u} := (Q^T Q)^{-1/2} u \\
& = & \bar{u}^T c + \dsp \min_{x} \sqrt{\|x\|_2^2 + s^2} - \bar{u}^T x \\
& & \mbox{by letting} \; x := c - \bar{z}.
\end{array}
\end{equation*}
Now consider the second term $\dsp \min_{x} \sqrt{\|x\|_2^2 + s^2} - \bar{u}^T x$. The optimal $x^*$ must be in the direction of $\bar{u}$. Letting $x := \alpha \bar{u}$, $\alpha \in \reals$, we have the expression
\[
\dsp \min_{\alpha \in \reals} \sqrt{\alpha^2 \| \bar{u} \|_2^2 + s^2} - \alpha \|\bar{u} \|_2^2
\]
When $\| \bar{u} \|_2 \geq 1$, the problem is unbounded below.
When $\| \bar{u} \|_2 < 1$, the optimal solution is $\alpha^* = \frac{s}{\sqrt{1 - \| \bar{u} \|^2_2}}$ and the optimal value is thus $s \sqrt{1 - \| \bar{u} \|^2_2}$. The closed-form expression for $f_1(u)$ is therefore:
\begin{equation} \label{eq:first_subproblem}
\begin{array}{lll}
f_1(u) & = & \dsp \bar{u}^T c + \dsp \min_{x} \sqrt{\|x\|_2^2 + s^2} - \bar{u}^T x \\
& = & \bar{u}^T c + s \sqrt{1 - \| \bar{u} \|^2_2} \\
& = & u^T (Q^T Q)^{-1/2} c + s \sqrt{1 - u^T (Q^T Q)^{-1} u} \\
& = & u^T K^{-1/2} c + s \sqrt{1 - u^T K^{-1} u} \\
& & \mbox{by letting} \; K := Q^T Q.
\end{array}
\end{equation}

\emph{Second subproblem.} Consider the function $f_2(u) := \dsp \min_{w \in \reals^n} \epsilon \| w \|_2 + \lambda \| w \|_1 - (P u)^T w$. We observe that the objective function is homogeneous. Strong duality gives:
\begin{equation} \label{eq:second_subproblem}
\begin{array}{lll}
f_2(u) & = & \dsp \min_{w} \max_{v, r} r^T w + v^T w - (P u)^T w \\
& & \mbox{s.t.} \quad \|r\|_2 \leq \epsilon, \| v \|_\infty \leq \lambda \\
& = & \dsp \max_{v, r} \min_{w} (r + v - Pu)^T w \\
& & \mbox{s.t.} \quad \|r\|_2 \leq \epsilon, \| v \|_\infty \leq \lambda \\
& = & \dsp \max_{v, r} 0 \\
& & \mbox{s.t.} \quad \|r\|_2 \leq \epsilon, \| v \|_\infty \leq \lambda, Pu = v + r \\
\end{array}
\end{equation}
Hence $f_2(u) = 0$ if there exists $v, r \in \reals^n$ satisfying the constraints. Otherwise $f_2(u)$ is unbounded below.

\emph{Dual problem.} From (\ref{eq:first_subproblem}) and (\ref{eq:second_subproblem}), the dual problem to (\ref{eq:robust-lasso}) can be derived as:
\begin{equation*}
\begin{array}{lll}
\dsp \phi_{\lambda, \epsilon} = & \dsp \max_{\substack{u \in \reals^k, \\ v, r \in \reals^n}} & u^T K^{-1/2} c + s \sqrt{1 - u^T K^{-1} u} \\
& \mbox{s.t.} & \| v \|_\infty \leq \lambda,\; \| r \|_2 \leq \epsilon, \; Pu = v + r
\end{array}
\end{equation*}
Letting $R := P K^{1/2} \in \reals^{n \times k}$ and replacing $u$ by $K^{-1/2} u$, we have
\begin{equation}
\label{eq:robust_lasso_dual}
\begin{array}{llll}
\dsp \phi_{\lambda, \epsilon} & = & \dsp \max_{\substack{u \in \reals^k, \\ v, r \in \reals^n}} & u^T c + s \sqrt{1 - u^T u} \\
& & \mbox{s.t.} & \| v \|_\infty \leq \lambda, \; \| r \|_2 \leq \epsilon, \; Ru = v + r \\
& = & \dsp \max_{u, v, r, t} & u^T c + s t \\
& & \mbox{s.t.} & \left \| \begin{bmatrix}
u\\
t
\end{bmatrix} \right \|_2 \leq 1, \; \| v \|_\infty \leq \lambda, \; \| r \|_2 \leq \epsilon, \\
& & & Ru = v + r \\
\end{array}
\end{equation}
\emph{Bidual problem.} The bidual of (\ref{eq:robust-lasso}) writes
\begin{equation*}
\begin{array}{lll}
\dsp \phi_{\lambda, \epsilon} & = & \dsp \min_{w \in \reals^n} \max_{u, v, r, t} u^T c + s t + w^T (v + r - R u) \\
& & \mbox{s.t.} \quad \left \| \begin{bmatrix}
u\\
t
\end{bmatrix} \right \|_2 \leq 1, \| v \|_\infty \leq \lambda, \| r \|_2 \leq \epsilon \\
& = & \dsp \min_{w \in \reals^n} \max_{u, v, r, t} \begin{bmatrix}
c - R^T w\\
s
\end{bmatrix} ^T \begin{bmatrix}
u\\
t
\end{bmatrix} + w^T v + w^T r \\
& & \mbox{s.t.} \quad \left \| \begin{bmatrix}
u\\
t
\end{bmatrix} \right \|_2 \leq 1, \| v \|_\infty \leq \lambda, \; \| r \|_2 \leq \epsilon
\end{array}
\end{equation*}

Therefore,
\begin{equation} \label{eq:bidual_robust_lasso}
\begin{array}{ll}
\dsp \phi_{\lambda, \epsilon} & = \dsp \min_{w \in \reals^n} \left \| \begin{bmatrix}
c - R^T w\\
s
\end{bmatrix} \right \|_2 + \epsilon \| w \|_2 + \lambda \| w \|_1
\end{array}
\end{equation}
where $R \in \reals^{n \times k}$, $c \in \reals^{k}$ and $s \in \reals$. Note that problem (\ref{eq:bidual_robust_lasso}) still involves $n$ variables in $w$, but now the size of the design matrix is only $n$-by-$k$, instead of $n$-by-$m$ as the original problem.

\emph{Summary.} To summarize, to solve problem~(\ref{eq:robust-lasso}) with $X = PQ^T$ we first set $c := (Q^T Q)^{-1/2} Q^T y$, $s := \sqrt{y^T y - c^T c}$, $R := P(Q^TQ)^{1/2}$, then solve the problem~(\ref{eq:bidual_robust_lasso}) above. As discussed later, the worst-case complexity grows as $O(kn^2+k^3)$. This is in contrast with the original problem~(\ref{eq:robust-lasso}) when no structure is exploited, in which case the complexity grows as $O(mn^2+m^3)$.

\section{Safe Feature Elimination}
\label{sec:safe-feature-elimination}
In this section we present a method to reduce the dimension of (\ref{eq:bidual_robust_lasso}) without changing the optimal value. Let us define $A := \begin{bmatrix}
R^T\\ 0 \end{bmatrix} \in \reals^{(k + 1) \times n}$ and $b := \begin{bmatrix} c\\ s \end{bmatrix} \in \reals^{k + 1}$, problem (\ref{eq:bidual_robust_lasso}) becomes:
\begin{equation} \label{eq:safe_primal}
\begin{array}{ll}
\dsp \phi_{\lambda, \epsilon} = \dsp \min_{w \in \reals^n} \| Aw - b \|_2 + \epsilon \| w \|_2 + \lambda \| w \|_1
\end{array}
\end{equation}
Problem (\ref{eq:safe_primal}) is equivalent to:
\begin{equation*}
\begin{array}{llll}
\dsp \phi_{\lambda, \epsilon} & = & \dsp \min_{w \in \reals^n} \max_{\alpha, \beta, \gamma} \alpha^T (Aw - b)  + \beta^T w + \gamma^T w \\
& & \mbox{s.t.} \quad \| \alpha \|_2 \leq 1, \, \| \beta \|_2 \leq \epsilon, \, \| \gamma \|_\infty \leq \lambda \\
& = & \dsp \max_{ \substack{\| \alpha \|_2 \leq 1 \\ \| \beta \|_2 \leq \epsilon \\ \| \gamma \|_\infty \leq \lambda}} \min_{w \in \reals^n} w^T (A^T \alpha + \beta + \gamma) - \alpha^T b \\
& = & \dsp \max_{\alpha, \beta, \gamma} - \alpha^T b \\
& & \mbox{s.t.} \quad \| \alpha \|_2 \leq 1, \, \| \beta \|_2 \leq \epsilon, \, \| \gamma \|_\infty \leq \lambda, \\
& & \qquad \; A^T \alpha + \beta + \gamma = 0 \\
& = & \dsp \max_{\alpha, \beta, \gamma} - \alpha^T b \\
& & \mbox{s.t.} \quad \| \alpha \|_2 \leq 1, \, \| \beta \|_2 \leq \epsilon, \\
& & \qquad \; | a_i^T \alpha + \beta_i | \leq \lambda, \forall i = 1 \ldots n
\end{array}
\end{equation*}
where $a_i$'s are columns of $A$. If $\|a_i\|_2 \leq \lambda - \epsilon$,  we always have $|a_i^T \alpha + \beta_i | \leq | a_i^T \alpha | + | \beta_i | \leq \lambda$. In other words, we can then safely discard the $i$-th feature without affecting our optimal value.

\section{Non-robust square-root LASSO}
\label{sec:non-robust-square-root-lasso}
In practice, a simple idea is to replace the data matrix by its low rank approximation in the model. We refer to this approach as the non-robust square-root LASSO:
\begin{equation}
\label{eqn:non-robust-square-root-lasso}
\begin{array}{ll}
\dsp \min_{w \in \reals^n} \| Q P^T w - b \|_2 + \lambda \| w \|_1
\end{array}
\end{equation}

\begin{figure}[t]
\centering
\subfigure[Non-robust rank-1 square-root LASSO.]{
\includegraphics[width=0.36\textwidth]{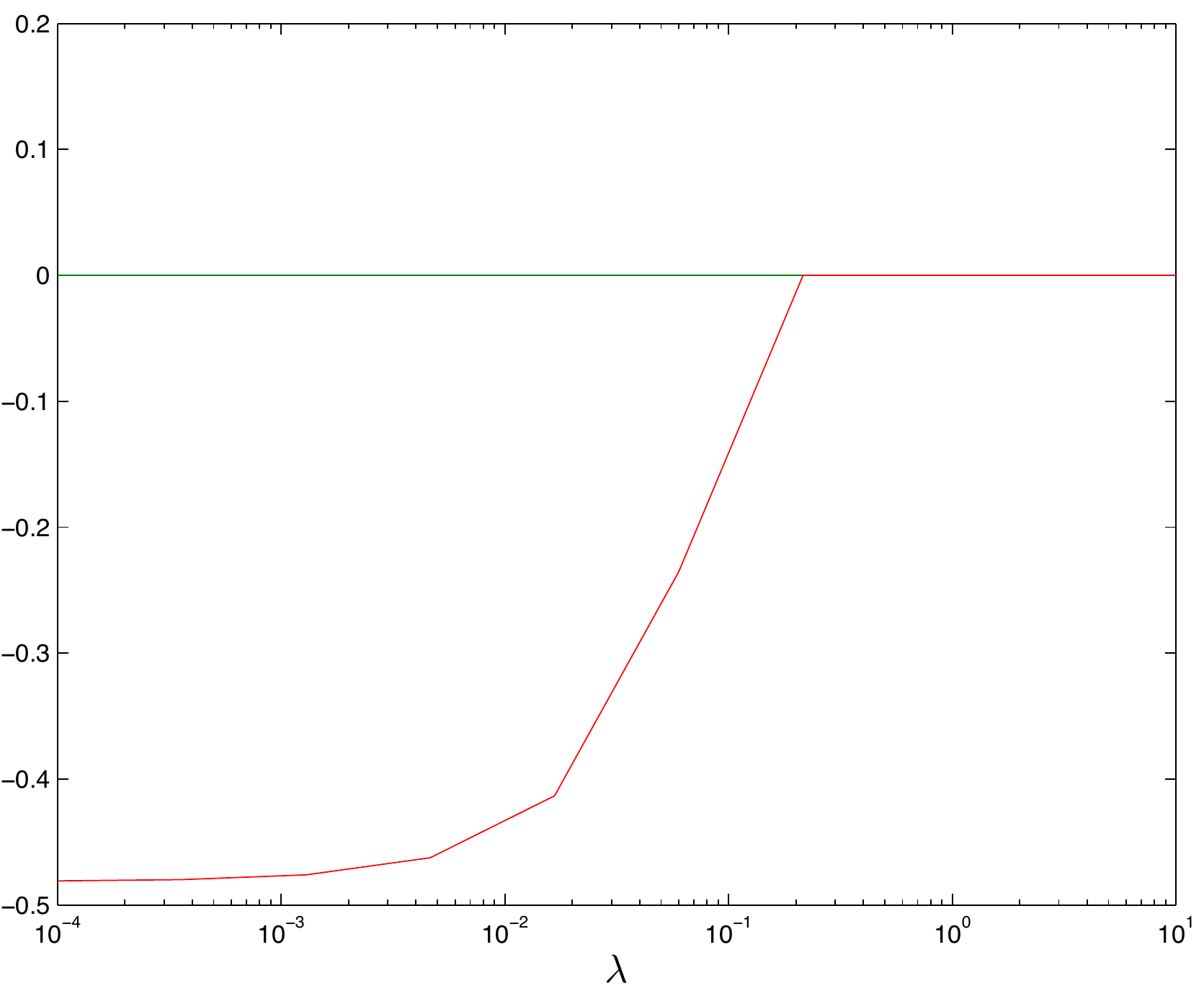}
\label{fig:graph-non-robust-sqrt-lasso}
}
\subfigure[Robust rank-1 square-root LASSO.]{
\includegraphics[width=0.36\textwidth]{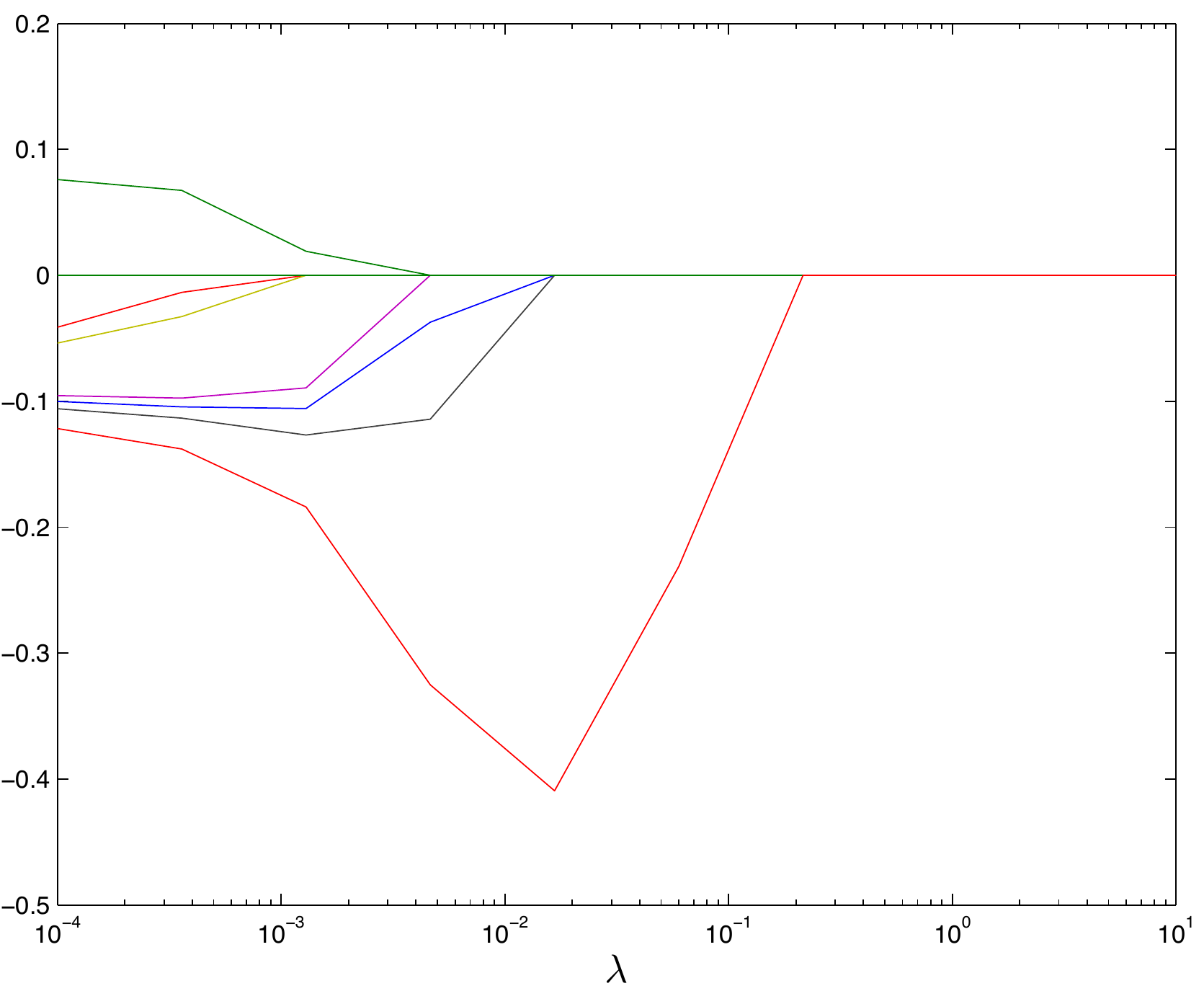}
\label{fig:graph-robust-sqrt-lasso}
}
\caption{Non-robust versus Robust square-root LASSO under rank-1 approximated data.}
\label{fig:non-robust-versus-robust}
\end{figure}

For many learning applications, this approach first appears as an attractive heuristic to speed up the computation. Nevertheless, in problems with sparsity as the main emphasis, care must be taken in the presence of the regularization involving $l_1$-norm. Figure~\ref{fig:graph-non-robust-sqrt-lasso} shows an example of a non-robust square-root LASSO with data replaced by its rank-$1$ approximation. The optimal solution then always has a cardinality at most $1$, and the tuning parameter $\lambda$ does not provide any sparsity control, unlike the robust low-rank model in Figure~\ref{fig:graph-robust-sqrt-lasso}. In general, replacing our data with its low-rank approximation will result in the lost of the sparsity control from regularization parameters. We provide an insight for this absence of sparsity control in the following theorem.

\begin{theorem}
\label{thm:non-robust-square-root-lasso}
For the non-robust square-root LASSO problem (\ref{eqn:non-robust-square-root-lasso}), with $P \in \reals^{n \times k}$ and $Q \in \reals^{m \times k}$ full rank where $k \ll \min \{ m, n \}$, there exists a LASSO solution with cardinality at most $k$.
\end{theorem}

\begin{proof}
Uniquely decomposing $b$ into $b = Qz + u$ where $u \perp \range(Q)$ gives
\begin{equation*}
\begin{array}{ll}
& \dsp \min_{w \in \reals^n} \| Q(P^Tw - z) - u \|_2 + \lambda \| w \|_1 \\
= & \dsp \min_{w \in \reals^n} \sqrt{\| Q(P^Tw - z)\|^2_2 + \| u \|^2_2} + \lambda \| w \|_1
\end{array}
\end{equation*}
Let $w_0$ be any optimal solution to this problem, it suffices to show that the problem
\begin{equation*}
\begin{array}{ll}
& \dsp \min_{w \in \reals^n} \| w \|_1 \; : \; P^T w = P^T w_0
\end{array}
\end{equation*}
has an optimal solution with cardinality at most $k$. We prove this in the following lemma:
\end{proof}

\begin{lemma}
The problem
\begin{equation}
\label{eqn:basis-pursuit}
\begin{array}{ll}
& \dsp \min_{x \in \reals^n} \| x \|_1 \; : \; Ax = b
\end{array}
\end{equation}
with $A \in \reals^{k \times n}$ wide ($k < n$) and $b \in \range(A)$ has an optimal solution with cardinality at most $k$.
\end{lemma}

\begin{proof}
Our proof is adapted from \cite{tibshirani2013lasso} on the existence and uniqueness of the solution. Let $x \in \reals^n$ be an optimal solution to (\ref{eqn:basis-pursuit}). Without loss of generality, we can assume all components of $x_i$ are non-zeros (if some components are zeros one can discard the corresponding columns of $A$).

If $\card(x) > k$, we provide a procedure to reduce the cardinality of $x$, while keeping the same $l_1$-norm and constraint feasibility. Let $s \in \reals^n$ be the (unique) subgradient of $\| x \|_1$: $s_i := \mbox{sign}(x_i), i = 1, \ldots, n$. The optimality condition of (\ref{eqn:basis-pursuit}) shows that $\exists \mu \in \reals^k$ such that $A^T \mu = s$. Since all the columns $A_i$'s are linearly dependent, there exist $i$ and $c_j$ such that
\begin{equation*}
\begin{array}{lll}
\dsp A_i & = & \sum_{j \in \Epsilon \backslash \{ i \}} c_j A_j, \text{ where } \Epsilon := \{1, \ldots, n \} \\
\dsp A_i^T \mu & = & \sum_{j \in \Epsilon \backslash \{ i \}} c_j A_j^T \mu \\
\dsp s_i \mu & = & \sum_{j \in \Epsilon \backslash \{ i \}} c_j s_j \\
\end{array}
\end{equation*}
Therefore $1 = s_i^2 = \sum_{j \in \Epsilon \backslash \{ i \}} c_j s_j s_i$. Defining $d_j := c_j s_j s_i$ gives $\sum_{j \in \Epsilon \backslash \{ i \} } d_j = 1$ and
\begin{equation*}
\begin{array}{lll}
\dsp s_i A_i & = & \sum_{j \in \Epsilon \backslash \{ i \}} c_j s_i A_j\\
\dsp & = & \sum_{j \in \Epsilon \backslash \{ i \}} c_j s_j s_i s_j A_j \text{ : since } s_j^2 = 1\\
\dsp & = & \sum_{j \in \Epsilon \backslash \{ i \}} d_j s_j A_j\\
\end{array}
\end{equation*}
Let us define a direction vector $\theta \in \reals^n$ as follows: $\theta_i := -s_i$ and $\theta_j := d_j s_j, j \in \Epsilon \backslash \{ i \}$. Then $A \theta = (-s_i A_i) + \sum_{j \in \Epsilon \backslash \{ i \}} d_j s_j A_j = 0$. Thus letting
\begin{equation*}
\begin{array}{ll}
x^{(\rho)} := x + \rho \theta \text { with } \rho > 0
\end{array}
\end{equation*}
we have $x^{(\rho)}$ feasible and its l-$1$ norm stays unchanged:
\begin{equation*}
\begin{array}{lll}
\| x^{(\rho)} \|_1 & = & |x_i^{(\rho)} | + \sum_{j \in \Epsilon \backslash \{ i \}} |x_j^{(\rho)} | \\
& = & (|x_i| - \rho) + \sum_{j \in \Epsilon \backslash \{ i \}} (x_j^{(\rho)} + \rho d_j ) \\
& = & \| x \|_1 + \rho \left ( \sum_{j \in \Epsilon \backslash \{ i \}} \rho d_j - 1 \right )\\
& = & \| x \|_1
\end{array}
\end{equation*}
Choosing $\rho := \min \{ t \geq 0 : x_j + t \theta_j = 0 \mbox{ for some } $j$ \}$
we have one fewer non-zeros components in $x^{(\rho)}$. Note that $\rho \leq |x_i|$. Therefore, repeating this process gives an optimal solution $x$ of at most $k$ non-zeros components.
\QED
\end{proof}

\textbf{Remark}. An alternative proof is to formulate problem (\ref{eqn:basis-pursuit}) as a linear program, and observe that the optimal solution is at a vertex of the constraint set. Our result is also consistent with the simple case when the design matrix has more features than observations, there exists an optimal solution of the LASSO problem with cardinality at most the number of observations, as shown by many authors (\cite{tibshirani2013lasso}).

\section{Theoretical Analysis}
\label{sec:complexity-analysis}
Our objective in this section is to analyze the theoretical complexity of solving problem (\ref{eq:safe_primal}):
\begin{equation*}
\begin{array}{ll}
\dsp \min_{x \in \reals^n} \| Ax - b \|_2 + \epsilon \| x \|_2 + \lambda \| x \|_1
\end{array}
\end{equation*}
where $A \in \reals^{k \times n}$, $b \in \reals^{k}$, and the optimization variable is now $x \in \reals^n$. We present an analysis for a standard second-order methods via log-barrier functions. We also remark that with a generic primal-dual interior point method, our robust model can effectively solve problems of $3 \times 10^5$ observations and $10^5$ features in just a few seconds. In practical applications, specialized interior-point methods for specific models with structures can give very high performance for large-scale problems, such as in \cite{kim2007interior} or in \cite{koh2007interior}. Our paper, nevertheless, does not focus on developing a specific interior-point method for solving square-root LASSO; instead we focus on a generalized model and the analysis of multiple instances with a standard method.

\subsection{Square-root LASSO}
In second-order methods, the main cost at each iteration is from solving a linear system of equations involving the Hessian of the barrier function (\cite{andersen2003implementing}). Consider the original square-root LASSO problem:
\begin{equation*}
\begin{array}{ll}
\dsp \min_{x \in \reals^n} \| Ax - b \|_2 = \dsp \min_{x \in \reals^n, \, s \in \reals} s \; : \; \| Ax - b \|_2 \leq s
\end{array}
\end{equation*}
The log-barrier function is $ \dsp \varphi_\gamma(x, s) = \gamma s - \log \left( s^2 - \| Ax - b \|^2_2 \right) $. The cost is from evaluating the inverse Hessian of
\begin{equation} \label{eq:log-barrier-function}
\begin{array}{ll}
f := - \log \left( s^2 - (Ax - b)^T (Ax - b) \right)
\end{array}
\end{equation}
Let $g := - \log \left( s^2 - w^T w \right)$, we have $\nabla g = \frac{2}{-g} \begin{bmatrix} -w\\s \end{bmatrix}$ and $\nabla^2 g = \frac{2}{-g} \begin{bmatrix} -I & 0 \\0 & 1 \end{bmatrix} + \nabla g (\nabla g)^T$. The Hessian $\nabla^2 g$ is therefore a diagonal plus a dyad.

For (\ref{eq:log-barrier-function}), rearranging the variables as $\tilde{x} = \begin{bmatrix} x\\s \end{bmatrix}$ gives $\begin{bmatrix} Ax - b\\s \end{bmatrix} = \begin{bmatrix} A & 0\\0 & 1 \end{bmatrix} \begin{bmatrix} x\\s \end{bmatrix} - \begin{bmatrix} b\\0 \end{bmatrix} = \tilde{A} \tilde{x} - \tilde{b}$ where $\tilde{A} \in \reals^{(k+1) \times n}$ and:
\begin{equation} \label{eq:log-barrier-sqrt-lasso}
\begin{array}{lll}
\nabla^2 f \\
= \tilde{A}^T \left( \frac{2}{-f} \begin{bmatrix} -I & 0 \\0 & 1 \end{bmatrix} + \frac{4}{f^2} \begin{bmatrix} -(Ax - b)\\s \end{bmatrix} \begin{bmatrix} -(Ax - b)\\s \end{bmatrix}^T \right) \tilde{A} \\
= \frac{2}{-f} \tilde{A}^T \begin{bmatrix} -I & 0 \\0 & 1 \end{bmatrix} \tilde{A} \\
+ \frac{4}{f^2} \left( \tilde{A}^T \begin{bmatrix} -(Ax - b)\\s \end{bmatrix} \right) \left( \tilde{A}^T \begin{bmatrix} -(Ax - b)\\s \end{bmatrix} \right)^T
\end{array}
\end{equation}
The Hessian $\nabla^2 f$ is therefore simply a $(k + 2)$-dyad.

\subsection{Regularized square-root LASSO}
For (\ref{eq:safe_primal}), decomposing $x = p - q$ with $p \geq 0, q \geq 0$ gives
\begin{equation*}
\begin{array}{lll}
\phi & = & \dsp \min_{w \in \reals^n} \| Ax - b \|_2 + \epsilon \| x \|_2 + \lambda \| x \|_1 \\
& = & \dsp \min_{\substack{p, q \in \reals^n, \\ s, t \in \reals}} s + \epsilon t + \lambda \left( \ones^T p + \ones^T q \right) \\
& & \mbox{s.t.} \quad \|p - q\|_2 \leq t, \; \|A(p - q) - b\|_2 \leq s, \\& & \qquad \; p \geq 0, q \geq 0
\end{array}
\end{equation*}
The log-barrier function is thus
\begin{equation*}
\begin{array}{ll}
\varphi_\gamma (p, q, s, t) = & \gamma \left( s + \epsilon t + \lambda \left( \ones^T p + \ones^T q \right) \right) \\
& - \log \left(t^2 - \|p - q\|_2^2 \right) \\
& - \log \left( s^2 - \| A(p - q) - b\|_2^2 \right) \\
& - \dsp \sum_{i = 1}^n \log(p_i) - \dsp \sum_{i = 1}^n \log(q_i) \\
& - \log(s) - \log(t).
\end{array}
\end{equation*}

\emph{First log term}. Let $l_1 := -\log \left(t^2 - \|p - q\|_2^2 \right)$. Rearranging our variables as $\tilde{x} = \begin{bmatrix} p_1, q_1, \cdots p_n, q_n, t\end{bmatrix}^T$, we have
\begin{equation*}
\begin{array}{lll}
\nabla l_1  & = & \frac{2}{-l_1} \begin{bmatrix} p_1 - q_1, q_1 - p_1, \cdots p_n - q_n, q_n - p_n, t\end{bmatrix}^T \\
\nabla^2 l_1 & = & \frac{2}{-l_1}
\begin{bmatrix}
B &  &  & \\
& \ddots  &  & \\
&  & B & \\
&  &  & 1
\end{bmatrix} + \nabla l_1 (\nabla l_1)^T
\end{array}
\end{equation*}
where there are $n$ blocks of $B := \begin{bmatrix} -1 & 1\\ 1 & -1 \end{bmatrix}$ in the Hessian $\nabla^2 l_1$.

\emph{Second log term}. Let $l_2 := -\log \left( s^2 - \| A(p - q) - b\|_2^2 \right)$. Keeping the same arrangement of variables $\tilde{x} = \begin{bmatrix} p_1, q_1, \cdots p_n, q_n, s\end{bmatrix}^T$, we have
\begin{equation*}
\begin{array}{ll}
\begin{bmatrix}
A(p - q) - b\\
s
\end{bmatrix} = \tilde{A} \tilde{x}
\end{array}
\end{equation*}
where $\tilde{A} \in \reals^{(k+1) \times (2n + 1)}$. Following (\ref{eq:log-barrier-sqrt-lasso}), we have the Hessian is a $(k + 2)$-dyad.

\emph{Third log term}. Let $l_3 := -\dsp \sum_{i = 1}^n \log(p_i) - \dsp \sum_{i = 1}^n \log(q_i) -\log(s) - \log(t)$. Every variable is decoupled; therefore the Hessian is simply diagonal.

\textbf{Summary}. The Hessian of the log barrier function $\varphi_\gamma (p, q, s, t)$ is a block diagonal plus a $(k + 2)$-dyad. At each iteration of second-order method, inverting the Hessian following the matrix inversion lemma costs $O(k n^2)$. For the original square-root LASSO problem (\ref{eq:original-lasso}), using similar methods will cost $O(m n^2)$ at each iteration (\cite{andersen2003implementing}).

\section{Numerical Results}
\label{sec:numerical-results}
In this section, we perform experiments on both synthetic data and real-life data sets on different learning tasks. The data sets \footnote{All data sets are available at \url{http://www.csie.ntu.edu.tw/~cjlin/libsvmtools/datasets/}.} are of varying sizes, ranging from small, medium and large scales (Table~\ref{tab:data_statistics}). To compare our robust model and the full model, we run all experiments on the same workstation at 2.3~GHz Intel core i7 and 8GB memory. Both models have an implementation of the generic second-order algorithm from Mosek solver (\cite{andersen2000mosek}). For low-rank approximation, we use the simple power iteration methods. To make the comparison impartial, we do not use the safe feature elimination technique presented in Section \ref{sec:safe-feature-elimination} in our robust model.

\begin{table}
\caption{Data sets used in numerical experiments.}
\label{tab:data_statistics}
\begin{center}
{\scriptsize
\begin{tabular}{|c|rrrr|}
\hline
Data set         & \#train   & \#test    & \#features & Type \\
\hline
Gisette          & 6,000   & 1,000   & 5,000      & dense     \\
20 Newsgroups & 15,935 &  3,993 &  62,061 & sparse \\
RCV1.binary       & 20,242  & 677,399 & 47,236     & sparse    \\
SIAM 2007 & 21,519 &  7,077  & 30,438 & sparse \\
Real-sim & 72,309 & N/A & 20,958 & sparse \\
NIPS papers      & 1,500   & N/A     & 12,419     & sparse    \\
NYTimes & 300,000 & N/A     & 102,660    & sparse    \\
\hline
Random 1 & 500 & N/A & 100 & dense \\
Random 2 & 625 & N/A & 125 & dense \\
Random 3 & 750 & N/A & 150 & dense \\
\ldots & \ldots & \ldots & \ldots & \ldots \\
Random 19 & 2750 & N/A & 550 & dense \\
Random 20 & 2875 & N/A & 575 & dense \\
Random 21 & 3000 & N/A & 600 & dense \\
\hline
\end{tabular}
}
\end{center}
\end{table}


\subsection{Complexity on synthetic data}
Our objective in this experiment is to compare the actual computational complexity with the theoretical analysis presented in Section~\ref{sec:complexity-analysis}. We generated dense and i.i.d. random data for $n = 100 \ldots 600$. At each $n$, a data set of size $5n$-by-$n$ is constructed. We keep $k$ fixed across all problem sizes, run the two models and compute the ratio between the running time of our model to that of the full model. The running time of our model is the \emph{total} computational time of the data sketching phase and the training phase. The experiment is repeated 100 times at each problem size. As Figure~\ref{fig:performance} shows, the time ratio grows asymptotically as $O(1/n)$, a reduction of an order of magnitude in consistent with the theorical result.

\begin{figure}[h]
\begin{center}
\includegraphics[width=0.40\textwidth]{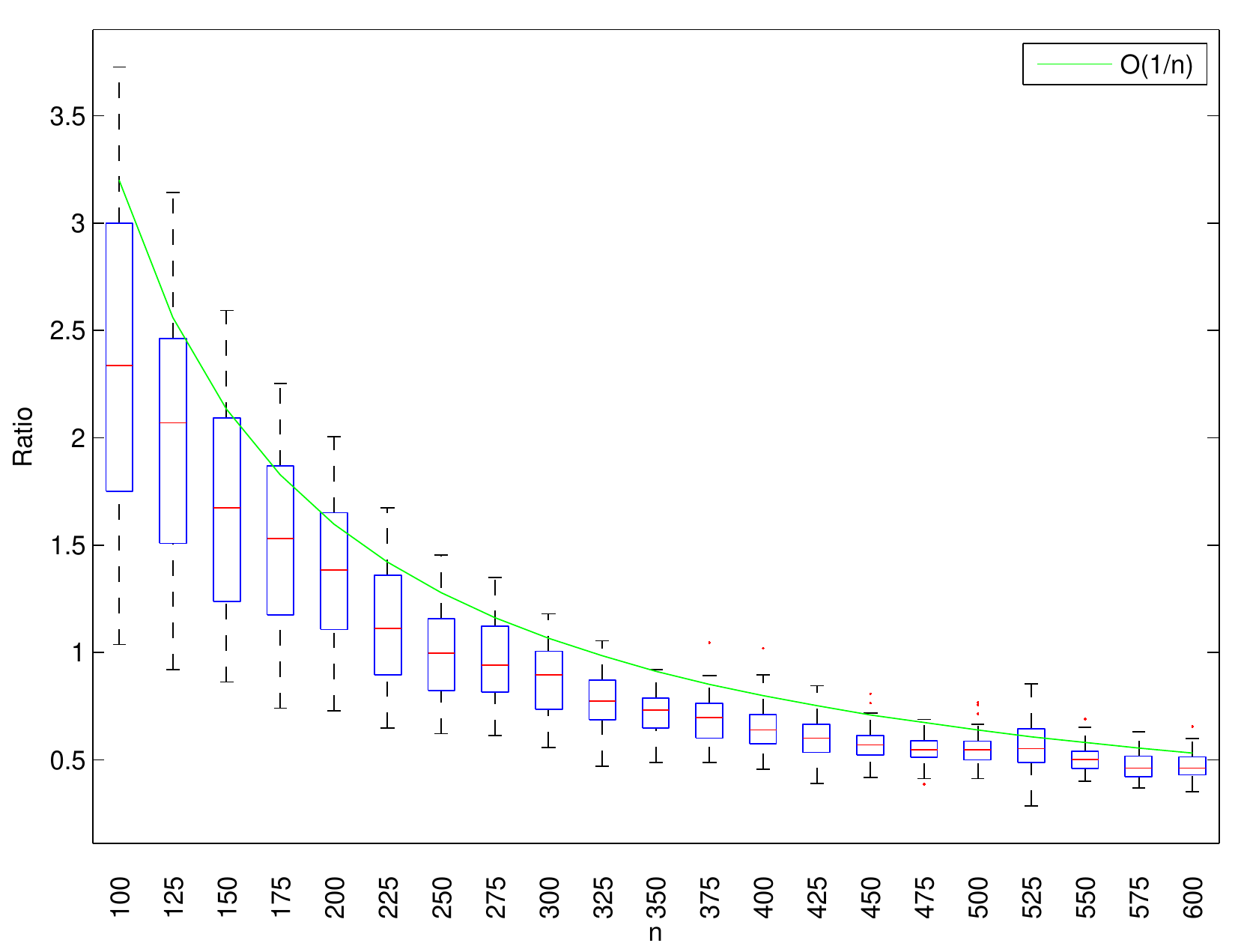}
\end{center}
\caption{The ratio between the running time of our robust model and the original model.}
\label{fig:performance}
\end{figure}

\subsection{Cross validation and leave-one-out}
In this experiment, we focus on the classical 5-fold cross validation on real-life data sets. Figure~\ref{fig:gisette_5_fold_cross_validation_time_against_k} shows the running time (in CPU seconds) from $k = 1 \ldots 50$ for 5-fold cross validation on Gisette data, the handwritten digit recognition data from NIPS 2003 challenge (\cite{guyon2004result}). It takes our framework less than 40 seconds, while it takes 22,082 seconds (500 times longer) for the full model to perform 5-fold cross validation. Furthermore, with leave-one-out analysis, the running time for the full model would require much more computations, becoming impractical while our model only needs a total of 7,684 seconds, even less than the time to carry out 5-fold cross validation on the original model. Table~\ref{tab:cross_validation} reports the experimental results on other real-life data sets.

\begin{figure}[h]
\begin{center}
\includegraphics[width=0.40\textwidth]{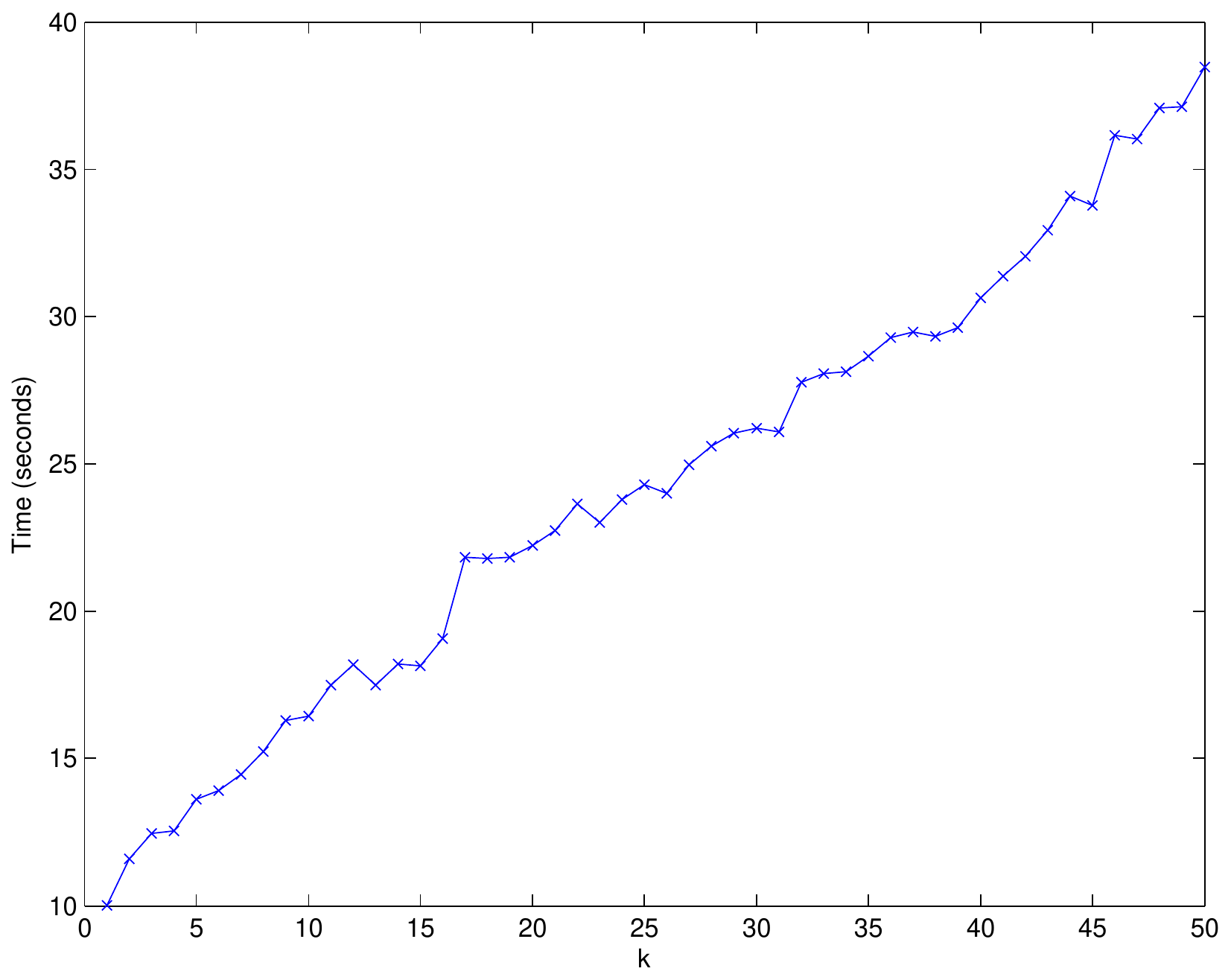}
\end{center}
\caption{5-fold cross validation on Gisette data with robust approximation model.}
\label{fig:gisette_5_fold_cross_validation_time_against_k}
\end{figure}

\begin{table}
\caption{Comparisons of 5-fold cross-validation on real data sets (in CPU time).}
\label{tab:cross_validation}
\begin{center}
{\scriptsize
\begin{tabular}{|c|rrr|}
\hline
Data set         & Original model   & Our model & Saving \\
& (seconds)   & (seconds) & factor \\
\hline
Gisette & 22,082 &  \textbf{39} & 566 \\
20 Newsgroups & 17,731 &  \textbf{65} & 272 \\
RCV1.binary & 17,776 &  \textbf{70.8} & 251 \\
SIAM 2007 & 9,025 &  \textbf{67} & 134 \\
Real-sim & 73,764 &  \textbf{56.3} & 1310 \\
\hline
\end{tabular}
}
\end{center}
\end{table}

\subsection{Statistical performance}
We further evaluate our model on statistical learning performance with binary classification task on both Gisette and RCV1 data sets. RCV1 is a sparse text corpus from Reuters while Gisette is a very dense pixel data. For evaluation metric, we use the F$1$-score on the testing sets. As Figure~\ref{fig:rcv1} and Figure~\ref{fig:gisette} show, the classification performance is equivalent to the full model. As far as time is concerned, the full model requires 5,547.1 CPU seconds while our framework needs 18 seconds for $k = 50$ on RCV1 data set. For Gisette data, the full model requires 991 seconds for training and our framework takes less than 34 seconds.

\begin{figure*}
\centering
\subfigure[Performance on binary classification.]{
\includegraphics[width=0.35\textwidth]{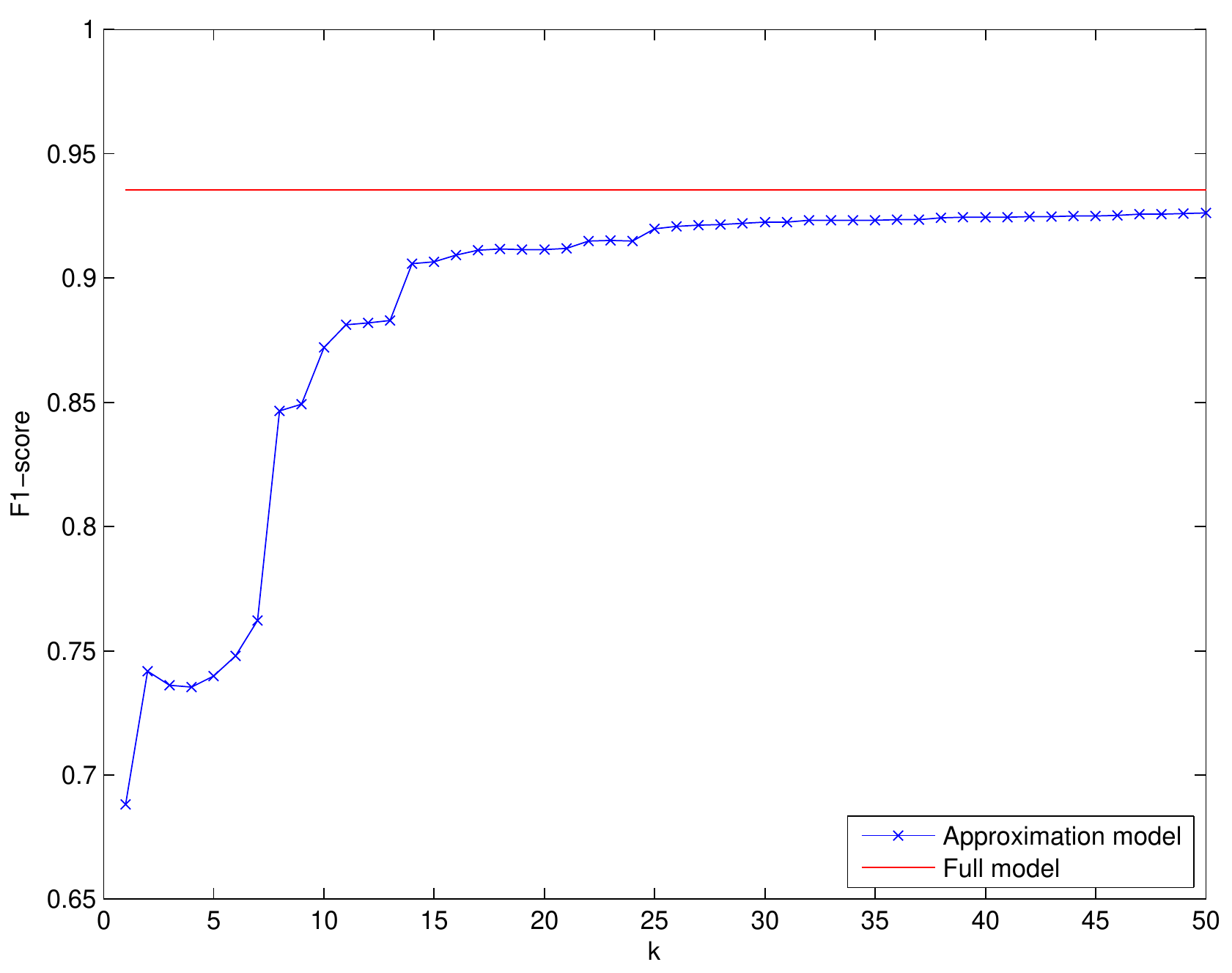}
\label{fig:rcv1_f1_against_k}
}
\subfigure[Training time of our model. The time to train the full model is 5547.1 seconds.]{
\includegraphics[width=0.35\textwidth]{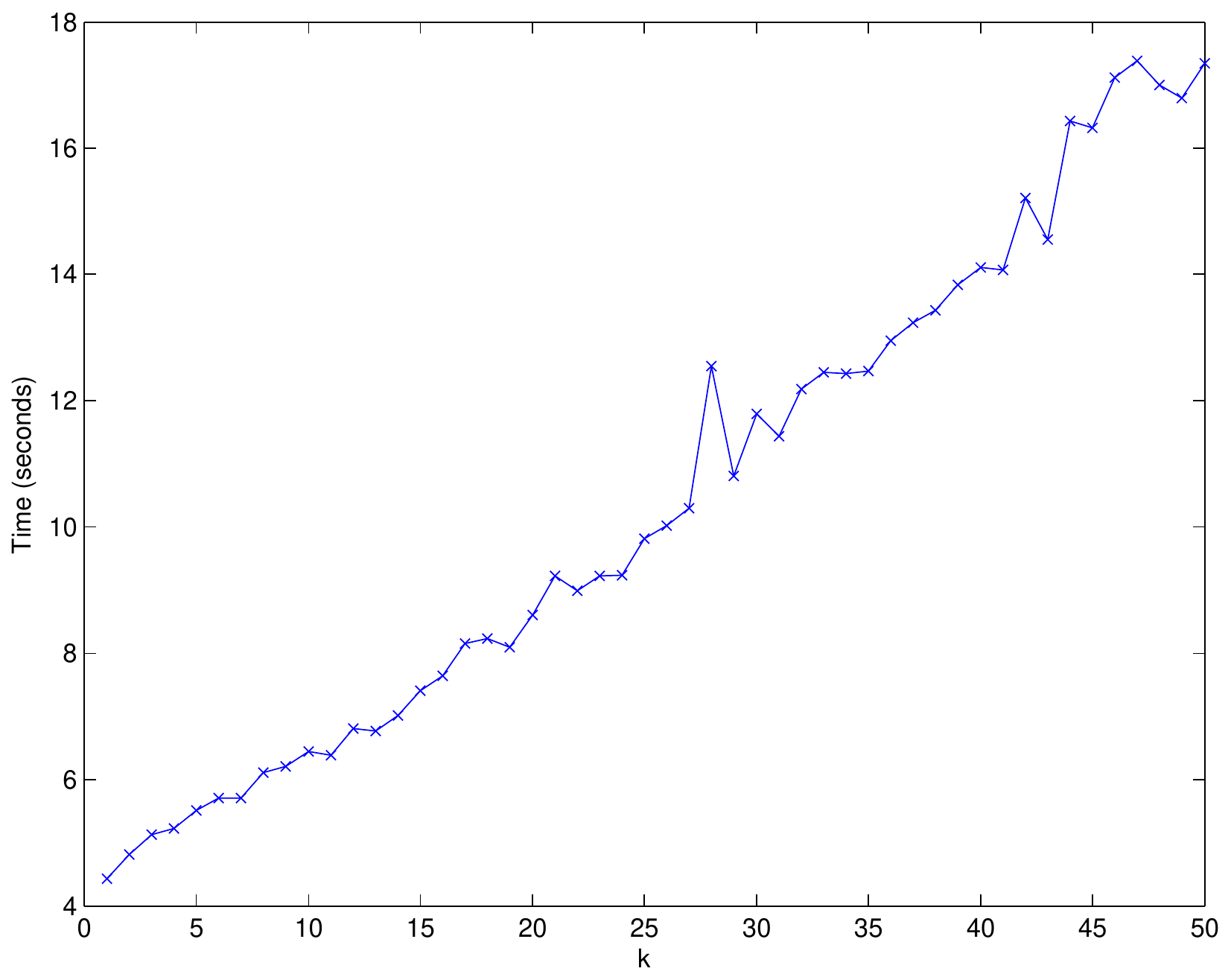}
\label{fig:rcv1_time_against_k}
}
\caption{Classification performance and running time on RCV1 data set.}
\label{fig:rcv1}
\end{figure*}

\begin{figure*}
\centering
\subfigure[Performance on binary classification.]{
\includegraphics[width=0.35\textwidth]{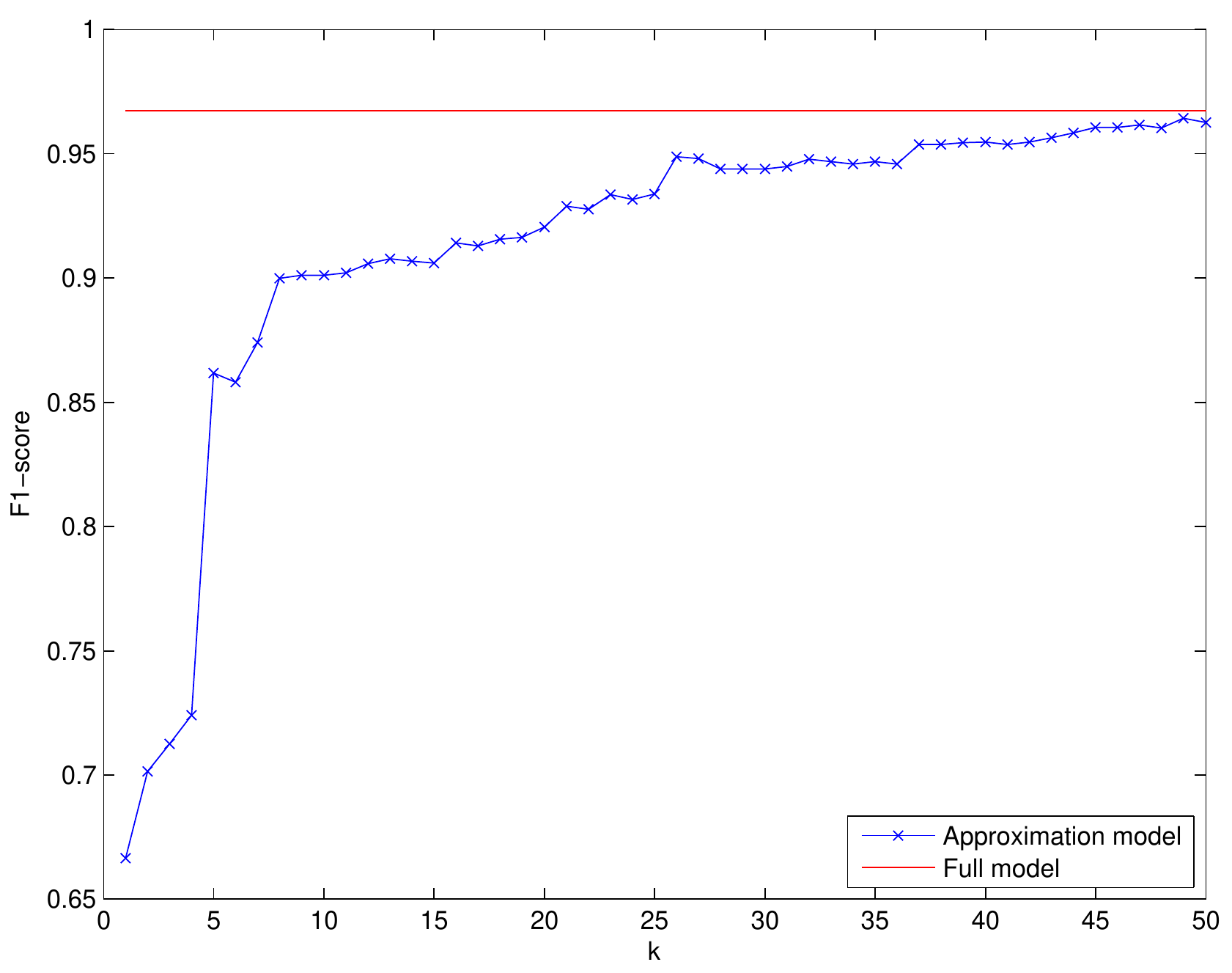}
\label{fig:gisette_f1_against_k}
}
\subfigure[Training time of our model. The time to train the full model is 991 seconds.]{
\includegraphics[width=0.35\textwidth]{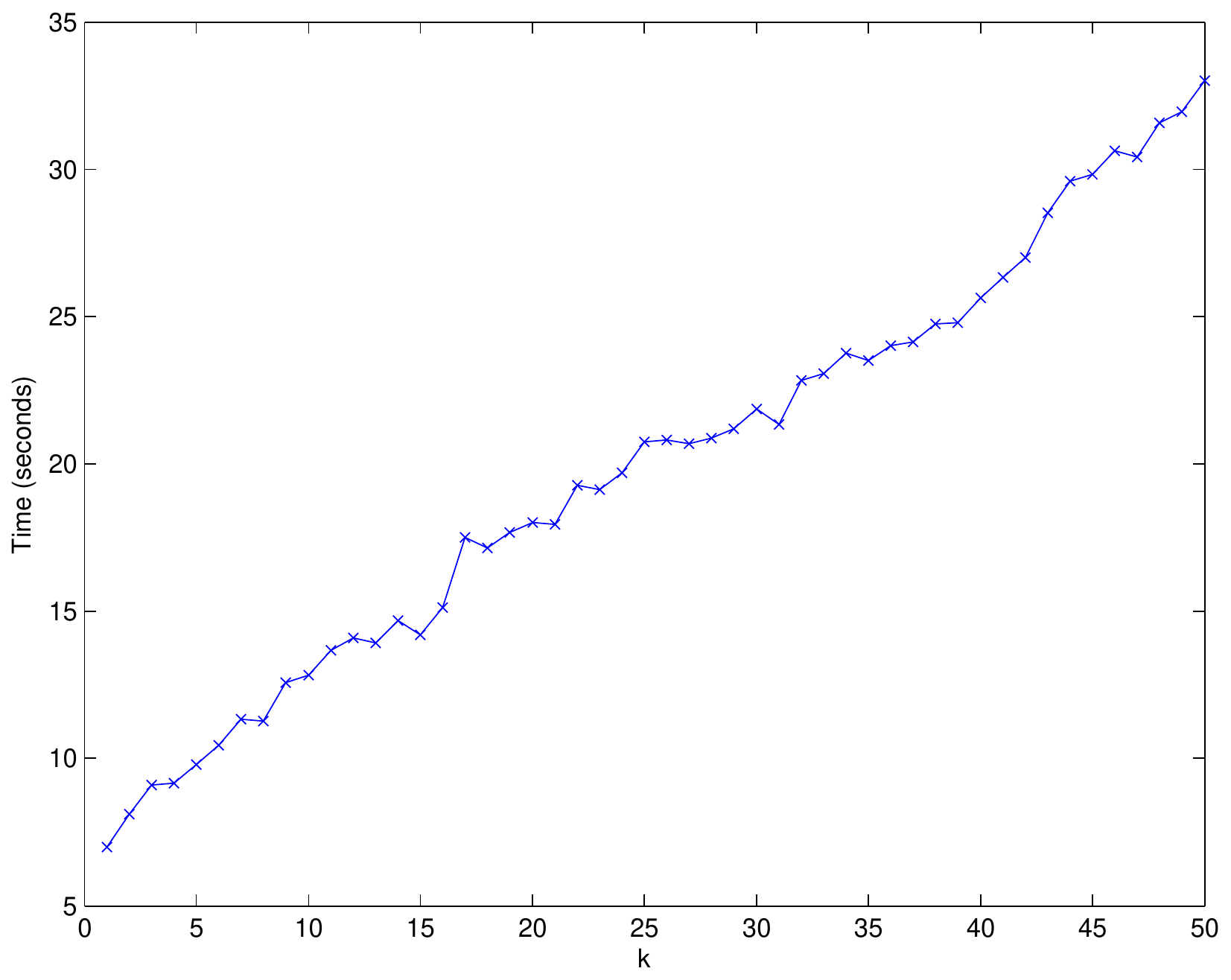}
\label{fig:gisette_time_against_k}
}
\caption{Classification performance and running time on Gisette data set.}
\label{fig:gisette}
\end{figure*}

\subsection{Topic imaging}
One application of solving multiple learning problems is topic imaging. Topic imaging is analogous in spirit to leave-one-out analysis. Instead of leave-one-observation-out, topic imaging removes a feature and runs a LASSO model on the remaining so as to explore the ``neighborhood'' (topic) of the query feature. Data sketching is computed only once for each data set and is shared to answer all queries in parallel.

We experiment our robust sketching model on two large text corpora: NIPS full papers and New York Times articles (\cite{Bache+Lichman:2013}). Table~\ref{tab:nips_papers_topic_imaging} and Table \ref{tab:nytimes_topic_imaging} show the results to sample queries on NIPS and NYTimes as well as the computational time our model takes to answer these queries. In both data sets, our model gives the result in just a few seconds. We can see the topic of Statistics, or Vision (Computer vision) with NIPS (Table~\ref{tab:nips_papers_topic_imaging}) and the theme of Political and Research with NYTimes data (Table~\ref{tab:nytimes_topic_imaging}).

\begin{table*}
\caption{Topic imaging for 5 query words on NIPS papers.}
\label{tab:nips_papers_topic_imaging}
\begin{center}
{\scriptsize
\begin{tabular}{r|lllll}
  Query & LEARNING & STATISTIC & OPTIMIZATION & TEXT & VISION \\
  \hline
  Time (s) & 3.15 & 2.89 & 3.11 & 3.52 & 3.15 \\
  \hline
  1 & error & data & algorithm & trained & object \\
  2 & action & distribution & data & error & image \\
  3 & algorithm & model & distribution & generalization & visiting \\
  4 & targeting & error & likelihood & wooter & images \\
  5 & weighed & algorithm & variable & classifier & unit \\
  6 & trained & parameter & network & student & recognition \\
  7 & uniqueness & trained & mixture & validating & representation \\
  8 & reinforced & likelihood & parame & trainable & motion \\
  9 & control & gaussian & bound & hidden & view \\
  10 & policies & set & bayesian & hmm & field
\end{tabular}
}
\end{center}
\end{table*}

\begin{table*}
\caption{Topic imaging for 5 query words on NIPS papers.}
\label{tab:nytimes_topic_imaging}
\begin{center}
{\scriptsize
\begin{tabular}{r|lllll}
  Query & HEALTH & TECHNOLOGY & POLITICAL & BUSINESS & RESEARCH \\
  \hline
  Time (s) & 11.81 & 11.84 & 10.95 & 11.93 & 10.65 \\
  \hline
  1 & drug & weaving & campaign & companionship & drug \\
  2 & patience & companies & election & companias & patient \\
  3 & doctor & com & presidency & milling & cell \\
  4 & cell & computer & vortices & stmurray & doctor \\
  5 & perceiving & site & republic & marker & percent \\
  6 & disease & company & tawny & customary & disease \\
  7 & tawny & wwii & voted & weaving & program \\
  8 & medica & online & democratic & analyst & tessie \\
  9 & cancer & sites & presidentes & firing & medical \\
  10 & care & customer & leader & billion & studly
\end{tabular}
}
\end{center}
\end{table*}

\section{Concluding Remarks}
We proposed in this paper a robust sketching model to approximate the task of solving multiple learning problems. We illustrate our approach with the square-root LASSO model given a low-rank sketch of the original data set. The numerical experiments suggest this framework is highly scalable, gaining one order of magnitude in computational complexity over the full model.

One interesting direction is to extend this model to a different data approximation, such as sparse plus low-rank (\cite{chandrasekaran2011rank}), in order to capture more useful information while keeping the structures simple in our proposed framework. Our future works also include an analysis and implementation of this framework using first-order techniques for very large-scale problems.

\bibliographystyle{abbrvnat.bst}
\bibliography{all_biblographies}

\end{document}

%% file: defs.tex
\newcommand{\BEAS}{\begin{eqnarray*}}
\newcommand{\EEAS}{\end{eqnarray*}}
\newcommand{\BEA}{\begin{eqnarray}}
\newcommand{\EEA}{\end{eqnarray}}
\newcommand{\BEQ}{\begin{equation}}
\newcommand{\EEQ}{\end{equation}}
\newcommand{\BIT}{\begin{itemize}}
\newcommand{\EIT}{\end{itemize}}
\newcommand{\BNUM}{\begin{enumerate}}
\newcommand{\ENUM}{\end{enumerate}}

\newcommand{\BA}{\begin{array}}
\newcommand{\EA}{\end{array}}

\newcommand{\ones}{\mathbf 1}

\newcommand{\reals}{{\mathbf{R}}}

\newcommand{\range}{{\mathcal R}}

\newcommand{\card}{\mathop{\bf card}}

\newcommand{\QED}{~~\rule[-1pt]{6pt}{6pt}}

\newcommand{\dsp}{\displaystyle}

\newtheorem{theorem}{Theorem}

\newtheorem{lemma}{Lemma}

\newcounter{exno}

\newenvironment{proof}{\emph{Proof.}}

{\begin{quote}}{\end{quote}}
